%% file: main_paper.tex
\documentclass[letterpaper,10pt,conference,romanappendices]{ieeeconf}
\IEEEoverridecommandlockouts                           

\usepackage{graphics} 
\usepackage{times} 
\usepackage{amsmath} 
\usepackage{amssymb}  
\usepackage{theorem}
\usepackage{color,xspace}
\usepackage{hyperref}
\usepackage{graphicx}
\usepackage{subcaption}
\usepackage{cite}
\usepackage{cleveref}
\usepackage{bm}
\usepackage{bbm}
\usepackage{tikz}
\usepackage{epstopdf} 
\usepackage{algorithm}
\usepackage{algpseudocode}
\algnewcommand{\LineComment}[1]{\State  #1}

\input{mysymbol.sty}

\newtheorem{theorem}{Theorem}

\newtheorem{remark}{Remark}
	
\newtheorem{corollary}{Corollary}
\newtheorem{proposition}{Proposition}	
\newtheorem{assumption}{Assumption}

\def\bL{\mathbf{L}}

\def\bx{\mathbf{x}}

\def\by{\mathbf{y}}

\def\R{\mathbb{R}}

\def\bnu{\boldsymbol{\nu}}

\title{A Chebyshev-Accelerated Primal-Dual Method for Distributed Optimization}
\author{Jacob H. Seidman$^1$, Mahyar Fazlyab$^2$, George J. Pappas$^2$, Victor M. Preciado$^2$
\thanks{1: Department of Applied Mathematics and Computational Science, University of Pennsylvania Email: seidj@sas.upenn.edu. 2: Department of Electrical and Systems Engineering, University of Pennsylvania. Email: \{mahyarfa, pappasg, preciado \}@seas.upenn.edu. Work was supported by the NSF under grants CAREER-ECCS-1651433 and IIS-1447470.}}

\begin{document}

\maketitle
\thispagestyle{empty}
\pagestyle{empty}

\begin{abstract}
We consider a distributed optimization problem over a network of agents aiming to minimize a global objective function that is the sum of local convex and composite cost functions. To this end, we propose a distributed Chebyshev-accelerated primal-dual algorithm to achieve faster ergodic convergence rates. In standard distributed primal-dual algorithms, the speed of convergence towards a global optimum (i.e., a saddle point in the corresponding Lagrangian function) is directly influenced by the eigenvalues of the Laplacian matrix representing the communication graph. In this paper, we use Chebyshev matrix polynomials to generate gossip matrices whose spectral properties result in faster convergence speeds, while allowing for a fully distributed implementation. As a result, the proposed algorithm requires fewer gradient updates at the cost of additional rounds of communications between agents. We illustrate the performance of the proposed algorithm in a distributed signal recovery problem. Our simulations show how the use of Chebyshev matrix polynomials can be used to improve the convergence speed of a primal-dual algorithm over communication networks, especially in networks with poor spectral properties, by trading local computation by communication rounds.
\end{abstract}

%
\section{Introduction}\label{se:introduction}

\input{introduction_revisions}

\section{Problem Formulation}

\input{problem_formulation_revisions}

\section{A Distributed Solution}
\input{distributed_solution_revisions}

\input{Numerical_Simulations_revisions}

\input{Conclusion_revisions}

\section{Appendix}

\subsection{Proof of Proposition \ref{proposition: kernel_check}} \label{proof: kernel_check proof}

\input{appendix}

\bibliographystyle{ieeetr}

\bibliography{references}

\end{document}

%% file: introduction_revisions.tex
In distributed consensus optimization, a network of agents aims to minimize a function constructed from individual (private) functions situated at the nodes of the network.  As the size of the network increases, it is important to develop fast and efficient distributed methods where agents can perform local computations to collectively find the global optimal solution.  This scenario arises in, for example, distributed learning problems \cite{yan2013distributed}, economic dispatch problems \cite{bai2017distributed}, and multi-agent coordination problems \cite{dimarogonas2012distributed}.

Methods to solve such distributed optimization problems fall into three main categories: \emph{(i)} primal methods, in which the problem is solved entirely in the primal domain and asymptotic consensus among agents is forced using an additional penalty function representing disagreements among agents \cite{shi2015extra, yuan2016convergence}; \emph{(ii)} primal-dual methods, in which the consensus constraint is relaxed by introducing Lagrangian multipliers and the agents iterate over the primal and dual variables to seek a saddle point of the Lagrangian \cite{Nedic2007, Feijer2010, Kia2015, Wang2010, Aybat2016}; and \emph{(iii)} dual methods, in which the constraints are encoded into a dual problem and the agents seek the optimum of a dual function instead \cite{Lin2014, Zargham2011, Tutunov2016,8431049}.  The advantages of each of these methods over the others are determined by what assumptions are made about the objective function (smoothness, strong convexity, easily computable proximal operators, etc.), as well as how any additional constraints are encoded in the problem.

In general, distributed first-order algorithms where agents update their local states based exclusively on the gradients of the cost functions can become slow for problems with ill-conditioned linear constraints.  In particular, the convergence speed of these algorithms is strongly conditioned by the second smallest eigenvalue of the Laplacian matrix of the communication graph.  The performance of distributed algorithms can be improved by preconditioning the matrix representing the topology of the communication graph.  For example, in \cite{Giselsson2014} the authors successfully applied this preconditioning technique to a dual gradient method.  In the numerical analysis literature, it has been known that condition numbers of linear operators can also be improved by applying specifically scaled and shifted Chebyshev polynomials \cite{Auzinger2011}.  The first applications of this idea to a linear consensus constraint can be found in \cite{Montijano2011} and \cite{Cavalcante2011}, where the use of Chebyshev polynomials gave a prescription for averaging information received from an agent's neighbors over multiple rounds of communication.  It was shown that this results in faster consensus over traditional choices of gossip matrices.  One interpretation of this method is that, by decreasing the condition number of the gossip matrix, we effectively create a more favorable connectivity structure of the network.

In \cite{Scaman2017}, the authors combined this idea with an accelerated dual ascent to create the \emph{multi-step dual accelerated} (MSDA) method for solving smooth and strongly convex distributed optimization problems.  By preconditioning the gossip matrix with a particular Chebyshev polynomial, the authors demonstrate that the effective condition number of the network decreases, and faster convergence for dual ascent is achieved.  In practice, when this algorithm is implemented, the cost of applying a Chebyshev polynomial of degree $K$ to the gossip matrix is that each iteration of the algorithm now requires $K$ rounds of communication.  When the cost of communication is cheap compared to the cost of computing gradients, their algorithm performs well even with the extra rounds of communication.

While the performance advantages of MSDA are clear, the performance guarantees of the algorithm only apply in the case of smooth and strongly convex objective functions.  Further, there is the hidden complexity of the computation of the dual gradients which either must be done analytically or through an inner minimization.  In cases where an analytic expression is not possible or the inner minimization is costly, a primal-dual scheme might be preferable.  

In this paper, we propose a Chebyshev-Accelerated Primal-Dual method for solving a distributed optimization problem with an objective function of the form $\sum_i f_i + \sum_i g_i$, where each $f_i$ is smooth and convex with Lipschitz gradient, and each $g_i$ is convex and non-differentiable.  Our algorithm is based on the primal-dual method proposed in \cite{Chambolle2015}, and later extended in \cite{Aybat2016}, which was developed to minimize a similar composite objective function.  In Theorem \ref{theorem: convergence_thrm}, we show that the largest possible primal and dual step sizes for this method are inversely related to the largest eigenvalue of the Laplacian matrix.  This motivates preconditioning the Laplacian matrix to reduce the spectral radius and allow the use of larger step sizes.  To this end, we combine the primal-dual method with an accelerated gossip strategy with Chebyshev polynomials from \cite{Montijano2011, Cavalcante2011, Scaman2017}.  Our algorithm reaches consensus more quickly than the standard primal-dual algorithm, and results in faster convergence as measured by the number of gradient updates.  We illustrate the performance of this algorithm with a distributed signal recovery problem and find that when the cost of communications are low compared to the cost of computing gradients, there is a noticeable increase in the convergence rate to the optimal value.

\subsection{Preliminaries}

We denote the set of real numbers by $\R$, the set of real $d$-dimensional vectors by $\R^d$, the $d$-dimensional vector of 1's and 0's by $1_d$ and $0_d$ respectively, and the $d$-dimensional identity matrix by $\operatorname I_d$.  A graph is defined as $\mathcal{G} = (\mathcal V, \mathcal E)$, where $\mathcal V$ is the set of vertices, and the edge set $\mathcal E$ contains pairs of vertices $(i,j)$ such that $(i,j) \in \mathcal E$ if and only if vertices $i$ and $j$ are connected in the graph.  Define the neighborhood set of a vertex $\mathcal N_i = \{ \ell \in \mathcal V\;|\;(i,\ell) \in \mathcal E \}$.  We assume here that $| \mathcal{V} | = n$ and the graph $\mathcal G$ is simple, connected, and undirected. We define the adjacency matrix $A = [a_{ij}]$ of $\mathcal G$ as $a_{ij} = 1$ if $(i,j) \in \mathcal E$ and $a_{ij} = 0$ otherwise.  The Laplacian matrix of $\mathcal G$ is defined as $\operatorname{L} = \operatorname{diag}(A1_n) - A$.  The matrix $\operatorname L$ is symmetric and positive semidefinite with eigenvalues $0 = \lambda_1 \leq \lambda_2 \leq \ldots \leq \lambda_n$, where the eigenvalue $\lambda_1=0$ corresponds to the eigenvector $1_n$.  If we take the eigenvalue decomposition $\operatorname L = U \operatorname{diag}(0, \lambda_2,\ldots \lambda_n)U^\top$ where $U$ is orthogonal, we may form the matrix $\sqrt{\operatorname{L}} = U \operatorname{diag}(0, \sqrt{\lambda_2}, \ldots, \sqrt{\lambda_n})U^\top$.  By $\| \operatorname L \|_2$ we denote the $\ell_2$ matrix norm (which coincides with the spectral radius for symmetric matrices), and we denote the kernel of the matrix $\operatorname L$ as $\ker(\operatorname L) = \text{span}\{1_n\}$.  The Kronecker product between two matrices or vectors is denoted $A \otimes B$ or $v \otimes w$, respectively.  A differentiable function $f: \R^n \to \R$ has a $L_f$-Lipschitz gradient if, for all $x,y \in \R^d$, we have
$$\|\nabla f(x) - \nabla f(y)\| \leq L_f \|x - y\|.$$ 
For a non-differentiable $g \colon \mathbb{R}^d \to \mathbb{R} \cup \{\infty\}$ with $\mathrm{dom} \, g = \{x \in \mathbb{R}^d \colon g(x)<\infty\}$, the subdifferential set is defined as
\begin{align*}
\partial g(x) := \{\gamma \in \R^d \colon g(x)+\gamma^\top (y-x) \leq g(y), \forall y\in \mathrm{dom}\,g\}.
\end{align*}

%% file: problem_formulation_revisions.tex
Consider a network of $n$ agents connected via a communication graph, $\mathcal{G} = (\mathcal V, \mathcal E)$, where the edge set $\mathcal E$ represents communication links between agents.  Each agent has access to private cost functions $f_i : \R^d \to \R$ and $g_i:\R^d \to \R \cup \{+\infty\}$. The agents seek to collectively find an optimal solution to the minimization problem
\begin{alignat}{2} 
\underset{x \in \mathbb{R}^d}{\mathrm{minimize}} \quad  && \sum_{i=1}^n \{f_i(x) + g_i(x)\}, \label{eq: global_problem}
\end{alignat} 
in a \emph{distributed} fashion. We make the following assumptions on the functions $f_i$ and $g_i$. 
\begin{assumption}\label{assumption: f and g}
The functions $f_i \colon \mathbb{R}^d \to \mathbb{R}$ are closed, convex, proper, and have $L_i$-Lipschitz gradients.  The functions $g_i \colon \mathbb{R}^d \to \mathbb{R} \cup \{\infty\}$ are proper, convex, lower semi-continuous.  The proximal maps
\begin{align}
\mathbf{prox}_{\alpha g_i}(x) := \underset{x' \in \R^d}{\argmin\;} \bigg\{g_i(x) + \frac{1}{2\alpha}\|x - x'\|_2^2\bigg\}.
\end{align}
are easily computable.
\end{assumption}
Examples of functions with an easily computable proximal operator are $g(x) = \|x\|_1$, $g(x)=\|x\|_\infty$ \cite{parikh2014proximal,duchi2008efficient}. Further, if $g(x)=\mathbb{I}_\mathcal{X}(x)$ is the indicator function of the closed convex set $\mathcal{X} \subset \mathbb{R}^d$, the proximal operator becomes a projection operator onto $\mathcal{X}$. In this case, we assume that projection onto $\mathcal{X}$ is easy to compute (e.g. box constraints).

We can reformulate problem \eqref{eq: global_problem} to be amenable to a distributed solution by introducing a set of variables $x_i \in \R^d$ for $i \in \mathcal V$, where $x_i$ corresponds to a local copy of the global state variable $x$ for agent $i$.  Let $\mathbf{x} = (x_1^\top,\ldots,x_n^\top)^\top$ be the concatenation of the local copies of the state variable.  If we define a gossip matrix $\mathbf{L} = \operatorname L \otimes \operatorname{I}_d$, using the Laplacian matrix $\operatorname L$ of the communication graph, we can ensure consensus (i.e. $x_i = x_j$ for all $i,j \in \mathcal V$) by adding the constraint $\sqrt{\mathbf{L}} \mathbf x = 0$, as $\ker(\sqrt{\operatorname L} \otimes \operatorname{I}_d) = \text{span}\{1_n \otimes x\;|\;x \in \R^d\}$.  Thus, the problem in \eqref{eq: global_problem} is equivalent to
\begin{alignat}{2} 
&\underset{\bx \in \mathbb{R}^{nd}}{\mathrm{minimize}} \quad  && f(\bx) + g(\bx) \label{eq: distributed_problem} \\
&\text{subject to}  && \quad \sqrt{\mathbf{L}}\bx = 0, \nonumber
\end{alignat} 
where $f \colon \mathbb{R}^{nd} \to \mathbb{R}$ and $g \colon \mathbb{R}^{nd} \to \mathbb{R} \cup \{\infty\}$ are now
\begin{align}
f(\bx) := \sum_{i=1}^n f_i(x_i), \quad g(\bx) := \sum_{i=1}^n g_i(x_i).
\end{align}
We form the augmented Lagrangian for \eqref{eq: distributed_problem},
\begin{equation} \label{eq: lagrangian}
\mathcal{L}(\bx,\by) := f(\bx)+g(\bx) + \by^\top \sqrt{\mathbf{L}}\bx + \frac{\rho}{2}\bx^\top \mathbf{L} \bx,
\end{equation}
where $\by = (y_1^\top, \ldots, y_n^\top)^\top$ is the stacked vector of local multipliers and $\rho \geq 0$ is the augmentation parameter.  Without loss of generality, we may assume that $\by \in (\text{ker}\sqrt{\mathbf{L}})^\perp$, as for any $\tilde \by \in \text{ker}\sqrt{\mathbf{L}}$ we have $(\by + \tilde \by)^\top\sqrt{\bL} = \by^\top \sqrt{\bL}$.
Since problem \eqref{eq: distributed_problem} has linear constraints and the feasible set is nonempty, Slater's condition is satisfied and strong duality holds \cite{boyd2004convex}.  Hence, we can solve \eqref{eq: distributed_problem} by solving the following saddle point problem:
\begin{align}
\underset{\by \in (\text{ker}\sqrt{\mathbf{L}})^\perp }{\max}\;\underset{\bx \in \R^{nd}}{\min}\; \mathcal{L}(\bx,\by) \label{eq: saddle_problem}.
\end{align}
A solution of the saddle point problem is a primal dual pair $(\bx^\star, \by^\star)$ such that for all $(\bx, \by)$,
\begin{align}
\mathcal L(\bx^\star, \by) \leq \mathcal L(\bx,\by) \leq \mathcal L(\bx, \by^\star).
\end{align}
Note that $\mathcal L(\bx, \by^\star) - \mathcal L(\bx^\star, \by)$ approaches 0 as $(\bx,\by) $ approaches  $(\bx^\star, \by^\star)$.  The optimality condition for an optimal primal-dual pair $(\bx^\star,\by^\star)$ is
\begin{subequations}
\begin{align}
0 &= \nabla f(\bx^\star) + T_g (\bx^\star) + \sqrt{\bL} \by^\star, \label{eq: opt_condition1}\\ 
0 &= \sqrt{\bL} \bx^\star,\label{eq: opt_condition2}
\end{align}
\end{subequations}
where $T_g (\bx^\star)$ is any subgradient of $g$ evaluated at $\bx^\star$. Note that, by the second equation, we have consensus, $\bx^\star = \mathrm{1}_n \otimes x^\star $. Multiplying both sides of \eqref{eq: opt_condition1} from the left by $\mathrm{1}_n^\top \otimes \operatorname I_d $, and noticing that $(\mathrm{1}_n^\top \otimes \operatorname I_d ) \ \sqrt{\bL}= (\mathrm{1}_n^\top \otimes \operatorname I_d ) (\sqrt{\mathrm{L}} \otimes \operatorname I_d ) = (\mathrm{1}_n^\top \sqrt{\mathrm{L}}) \otimes (\operatorname I_d  \operatorname I_d )=\mathrm{0}_{n}^\top \otimes \operatorname I_d $, we obtain 
\begin{align*}
\sum_{i=1}^{n} \{\nabla f_i(x^\star)+T_{g_i}(x^\star)\}=0, \quad T_{g_i}(x^\star) \in \partial g_i(x^\star),
\end{align*}
which is the optimality condition for the centralized problem in \eqref{eq: global_problem}.  Hence, the solution of the saddle point problem for the decentralized problem in \eqref{eq: distributed_problem} provides a solution of the centralized problem \eqref{eq: global_problem}.

%% file: distributed_solution_revisions.tex
%
%

\subsection{The Primal-Dual Algorithm}

We begin our exposition by presenting a primal-dual iterative algorithm proposed by Chambolle and Pock \cite[Algorithm 1]{Chambolle2015}, which is able to solve problems of the form \eqref{eq: saddle_problem}.  Writing $\nabla f(\bx) = (\nabla f_1(x_1)^\top, \ldots, \nabla f_n(x_n)^\top)^\top$, the iterations of this algorithm are given by
\begin{subequations} \label{eq: primal_dual_before 0}
	\begin{align}
	\hat\bx^{k+1} &= \bx^k -\alpha\big(\nabla f(\bx^k) + \sqrt{\mathbf{L}}\by^k + \rho \mathbf{L} \bx^k \big),  \label{eq: primal_dual_before xbar}\\
	\bx^{k+1} &= \argmin_{\bx} \left\{ g(\bx) \!+\! \frac{1}{2\alpha}\|\bx - \hat\bx^{k+1}\|_2^2 \right\},  \label{eq: primal_dual_before x}\\
	\by^{k+1} &= \by^k + \beta\sqrt{\mathbf{L}}(2\bx^{k+1}-\bx^k) \label{eq: primal_dual_before y},
	\end{align}
\end{subequations}
where we have interpreted the augmentation term as being absorbed into the smooth part of the objective function.  The parameters $\alpha, \beta >0$ are the primal and dual step sizes, respectively,  The first two updates constitute a proximal gradient step applied to the Lagrangian in the primal domain, while the third recursion is a gradient ascent-like step in the dual domain.  

The following theorem shows that under appropriate selection of the step sizes, the iterates in \eqref{eq: primal_dual_before 0} achieve an ergodic $O(1/N)$ convergence rate.  An important fact for our development is that the maximum possible step sizes for convergence are inversely related to the largest eigenvalue of the Laplacian matrix, as stated below.
\begin{theorem} \label{theorem: convergence_thrm}
Consider the optimization problem in \eqref{eq: distributed_problem}, where $f_i$ and $g_i$ satisfy Assumption \ref{assumption: f and g}. Let the step sizes $\alpha$ and $\beta$ for the recursions in \eqref{eq: primal_dual_before 0} satisfy
\begin{align} \label{eq: step_size}
\frac{1}{\beta+\rho}\bigg(\frac{1}{\alpha} - L_f\bigg) \geq \lambda_{n},
\end{align}
where $L_f = \max_i L_i$. Let $\{(\bx^k,\by^k)\}_{k=1}^N$ be the sequence of points produced by Algorithm \ref{Algorithm: Vanilla_PD} with $\bx^0 \in \ker \sqrt{\mathbf{L}}$ and $\by^0 \in \ker(\sqrt{\mathbf{L}})^\perp$.  Then we have
\begin{align} 
\mathcal{L}(\bar \bx^N,\by^\star) \!-\! \mathcal{L}(\bx^\star,\bar \by^N)\! \leq \! \frac{1}{N}\bigg(\!\frac{\|\bx^\star\! -\! \bx^0\|_2}{2\alpha} \!+\! \frac{\|\by^\star \!-\! \by^0\|_2}{2\beta}\!\bigg), \label{eq: convergence_rate}
\end{align}
where $\bar \bx^N = \frac{1}{N} \sum_{k=1}^N \bx^k$ and $\bar \by^N = \frac{1}{N} \sum_{k=1}^N \by^k$.
\end{theorem} 
\begin{proof}
First note that $L_f = \max \{L_1, \ldots, L_n\}$ is an upper bound for the Lipschitz constant for $\nabla f(\bx)$.  Indeed, 
\begin{align}
\big\| \nabla f(\bx) - \nabla f(\by) \big\|_2^2 &= \sum_{i=1}^n \| \nabla f_i(x_i) - \nabla f_i(y_i)\|_2^2 \nonumber \\ 
&\leq \sum_{i=1}^n L_i^2 \|x_i - y_i\|_2^2 \\
&\leq L_f^2 \| \bx - \by \|_2^2. \nonumber
\end{align} 
Again, we will interpret the augmentation term as being absorbed into the smooth term of the objective function.  Thus, an upper bound for the Lipschitz constant for $\nabla\big(f(\bx)+ (\rho/2) \bx^\top \mathbf{\operatorname L} \bx\big)$ is $L_F := L_f + \rho \lambda_{n}$.  

To prove the theorem we appeal to \cite[Theorem 1.1]{Aybat2016}, where the problem in question is similar to this paper, but with an additional private cone constraint for each agent.  Since we have no such constraints, in our scenario the theorem states if 
\begin{align}\label{eq: Q condition}
\begin{bmatrix} \big(\frac{1}{\alpha} - L_F\big)\operatorname{I}_{nd} & 0 & -\sqrt{\mathbf{L}} \\ 0 & 0 & \operatorname{0} \\ -\sqrt{\mathbf{L}} & \operatorname{0} & \frac{1}{\beta}\operatorname{I}_{nd}\end{bmatrix} \succeq 0, 
\end{align}
then we have for all $(\bx,\by)$
\begin{align} \label{eq: Chambolle convergence}
\mathcal{L}(\bar \bx^N,\by) \!-\! \mathcal{L}(\bx,\bar \by^N) \leq  \frac{1}{N}\bigg(&\!\frac{1}{2\alpha}\|\bx - \bx^0\|_2 \!+\! \frac{1}{2\beta}\|\by - \by^0\|_2 \nonumber \\
& + (\by - \by^0)^\top \sqrt{\mathbf{L}} (\bx - \bx^0)\bigg).
\end{align}
From the Schur complement, \eqref{eq: Q condition} is equivalent to the condition that $\big(\frac{1}{\alpha} - L_F\big)\operatorname{I}_{nd} - \beta \mathbf{L} \succeq 0$, or
\begin{align}
\frac{1}{\beta}\bigg(\frac{1}{\alpha} - L_F\bigg)\operatorname{I}_{nd} \succeq \mathbf L, \label{eq: Schur_condition}
\end{align}
Since the $\mathbf{L}$ and $\operatorname L$ share the same eigenvalues, and the largest eigenvalue of $\operatorname L$ is equal to $\lambda_{n}$, if 
\begin{equation} \label{eq: old_step_size}
\frac{1}{\beta}\bigg(\frac{1}{\alpha} - L_F\bigg)\geq \lambda_{n}
\end{equation}
then $\eqref{eq: Schur_condition}$ and in turn \eqref{eq: Q condition} are satisfied, which by Theorem 1.1 of \cite{Aybat2016} implies \eqref{eq: Chambolle convergence}.  Recalling that $L_F = L_f + \rho \lambda_{n}$, a simple rearrangement of \eqref{eq: old_step_size} gives \eqref{eq: step_size}.

If we choose $\bx^0$ such that $\sqrt{\mathbf{L}}\bx^0 = 0$ (i.e., initialize $\bx^0$ to have consensus) and take $(\bx,\by) = (\bx^\star, \by^\star)$ in \eqref{eq: Chambolle convergence}, then $\sqrt{\mathbf{L}}(\bx^\star - \bx^0) = 0$, and we conclude
$$\mathcal{L}(\bar \bx^N,\by^\star) -\! \mathcal{L}(\bx^\star,\bar \by^N) \!\leq \! \frac{1}{N}\bigg(\!\frac{\|\bx^\star - \bx^0\|_2}{2\alpha} + \frac{\|\by^\star - \by^0\|_2}{2\beta}\bigg),$$
as desired.
\end{proof}

\begin{remark}
If we do not require $\bx_0 \in \ker \sqrt{\mathbf{L}}$, the convergence rate holds with an additional constant factor of $(\by^\star - \by_0)^\top\sqrt{\mathbf{L}}(\bx^\star - \bx_0)$ in the parenthesis in \eqref{eq: convergence_rate}. 
\end{remark}

\begin{corollary}\label{corollary: x_rate}
If the dual step size is chosen as 
\begin{align}
\beta &= \frac{1}{\lambda_{n}}\bigg(\frac{1}{\alpha} - L_f \bigg) - \rho,
\end{align}
then we have the following ergodic convergence rate toward consensus
\begin{align} \label{eq: primal_rate}
\|\bar \bx^N - \bar \bx_C^N\|_2^2 \leq& \frac{1}{N}\bigg[ \frac{L_f + \rho \lambda_{n}}{\lambda_{2}}\|\bx^0 - \bx^\star\|_2 \nonumber \\
&+ \frac{\lambda_{n}}{\lambda_{2}} \frac{1}{L_f + \rho \lambda_{n}}\|\by^0 - \by^\star\|_2\bigg],
\end{align}
where 
$$\bar \bx_C^N = \sum_{i=1}^d(1 \otimes e_i)^\top \bar \bx^N(1 \otimes e_i),$$
is the projection of $\bar \bx^N$ onto the consensus subspace.
\end{corollary}
\begin{proof}
From the definition of the augmented Lagrangian, \eqref{eq: lagrangian}, we write
\begin{align}
\mathcal{L}(\bar \bx^N,\by) \!-\! \mathcal{L}(\bx,\bar \by^N) =& f(\bar \bx^N) \!-\! f(\bx^\star) + g(\bar \bx^N) \!-\! g(\bx^\star) \nonumber\\
 &+ \by^\star \sqrt{\mathbf{L}}\bar\bx^N + (\rho/2)(\bar \bx^N)^\top \mathbf{L} \bar \bx^N
 \end{align}
 Using \eqref{eq: opt_condition1} and \eqref{eq: opt_condition2}, we have the relation
 $$\by^\star \sqrt{\mathbf{L}}\bar\bx^N = -(\nabla f(\bx^\star)+T_g(\bx^\star))^\top(\bar\bx^N \!- \!\bx^\star)$$
 Therefore we may write
\begin{align}
\mathcal{L}(\bar \bx^N,\by) \!-\! \mathcal{L}(\bx,\bar \by^N) =& A + B + C,
\end{align}
where
\begin{subequations}
\begin{align}
A &= f(\bar \bx^N) \!-\! f(\bx^\star) \!- \!\nabla f(\bx^\star)^\top(\bar\bx^N \!- \!\bx^\star) \\ 
B &= g(\bar \bx^N) \!- \!g(\bx^\star) \!- \!T_g(\bx^\star)^\top(\bar\bx^N \!- \!\bx^\star) \\
C &= \frac{\rho}{2}(\bar \bx^N)^\top \mathbf{L} \bar \bx^N
\end{align}
\end{subequations}
By the convexity of $f$ and $g$, both $A,B \geq 0$.  Since $\ker \mathbf{L} = \{ 1 \otimes x\;|\;x \in \R^d\}$, the value of $(\bar \bx^N)^\top \mathbf{L} \bar \bx^N$ only depends on the displacement of $\bar \bx^N$ to its projection onto the consensus subspace.  Since $\mathbf{L}$ has a trivial kernel on this subspace, we may make the lower bound,
$$C = \frac{\rho}{2}(\bar \bx^N)^\top \mathbf{L} \bar \bx^N \geq \frac{\rho \lambda_2(\operatorname L)}{2} \| \bar \bx^N - \bar \bx_C^N\|_2^2.$$
Hence we have proven
$$\mathcal{L}(\bar \bx^N,\by) \!-\! \mathcal{L}(\bx,\bar \by^N)\geq \frac{\rho \lambda_2(\operatorname L)}{2} \| \bar \bx^N - \bar \bx_C^N\|_2^2.$$
The result now follows from combining this bound with \eqref{eq: convergence_rate}.
\end{proof}

Although the Laplacian matrix $\mathbf L$ preserves the sparsity pattern of the communication graph, the matrix $\sqrt{\mathbf L}$ does not; hence, the iterations in \eqref{eq: primal_dual_before 0} cannot be implemented in a distributed fashion.  To move to a fully distributed implementation, we make the change of variables $\bnu = \sqrt{\mathbf{L}}\by + \rho \mathbf{L} \bx$.  The iterations in \eqref{eq: primal_dual_before 0} can now be written in terms of the new variable $\bnu$, as follows:
\begin{subequations} \label{eq: primal_dual 0}
	\begin{align}
	\hat\bx^{k+1} &= \bx^k -\alpha\big(\nabla f(\bx^k) + \bnu^k\big),  \label{eq: primal_dual xbar}\\
	\bx^{k+1} &= \argmin_{\bx} \bigg\{g(\bx) + \frac{1}{2\alpha}\|\bx - \hat\bx^{k+1}\|_2^2\bigg\}.  \label{eq: primal_dual x}\\
	\bnu^{k+1} &= \bnu^k + \mathbf L \big((\rho + 2\beta)\bx^{k+1}-(\rho + \beta)\bx^k\big) \label{eq: primal_dual nu}.
	\end{align}
\end{subequations}

Here, we see why $\sqrt{\mathbf L}$ was chosen to create the consensus constraint instead of $\mathbf L$; we have been able to decouple the primal update from the graph structure and only the dual update must use communication between agents.  The update \eqref{eq: primal_dual nu} requires agents to know the states of their neighbors and the updates \eqref{eq: primal_dual xbar} and \eqref{eq: primal_dual x} can be performed by every agent with no communication at all.  Therefore, we may distribute the updates across individual agents as outlined in the following algorithm.

\begin{algorithm}[h]
	{\footnotesize \vspace*{1ex}
		\textbf{Given}: $f = \sum_{i=1}^n f_i$, where $f$ has $L_f$-Lipschitz gradient, $g = \sum_{i=1}^n g_i$.  An undirected and connected communication graph with Laplacian matrix $\operatorname L = [\operatorname L_{ij}]$, step sizes $0<\alpha,\beta$, augmentation parameter $\rho > 0$.
		\begin{algorithmic}[1]
			\State Initialize at $x_i^0$ such that $x_i^0 = x_j^0$ for all $i,j =1,\ldots,n$, initialize $\nu_i^0$ such that $\sum_i \nu_i^0 = 							0$			
			\For {$k=0,1,2,\cdots$ all agents}
			\State 
			$
			\hat x_i^{k+1} = x_i^k -\alpha\big(\nabla f_i(x_i^k) + \nu_i^k\big).
			$\vspace{.3 mm}
			\State
			$
			x_i^{k+1} = \argmin_x \{g_i(x) + \frac{1}{2\alpha}\|x - \hat x_i^{k+1}\|_2^2\}
			$\vspace{.3 mm}
			\State $\nu_i^{k+1}  = \nu_i^k + \sum_{\ell \in \mathcal{N}_i} \operatorname L_{i\ell}\big((\rho+2\beta)x_\ell^{k+1}-(\rho + \beta)x_\ell^k\big).$
			\EndFor
	\end{algorithmic}}
	\caption{\hspace*{-.5ex} Distributed Primal-Dual Algorithm} \label{Algorithm: Vanilla_PD}
\end{algorithm}

Note that since Theorem 1 assumes we initialize $\bx^0$ to have consensus, and we assume $\by^0 \in (\ker \sqrt{\mathbf L})^\perp$, we have $\bnu^0 = \sqrt{\mathbf{L}}\by \in (\ker \sqrt{\mathbf L})^\perp$, or equivalently, $\sum_i \nu_i^0 = 0$.

Theorem 1 and Corollary show the dependence of the convergence rate toward optimality on the largest eigenvalue $\lambda_{n}$ and the dependence on the convergence rate toward consensus on the condition number $\lambda_{n}/\lambda_{2}$ (see \eqref{eq: step_size} and \eqref{eq: primal_rate}).  It is here that we draw inspiration from \cite{Scaman2017} and in the next section, apply Chebyshev polynomials to $\mathbf{L}$ in order to increase the largest eigenvalue, reduce the condition number, and achieve a faster ergodic convergence rate.

\subsection{The Preconditioned Primal-Dual Algorithm}

The Chebyshev polynomials of the first kind are defined by the recurrence relation,
$$T_0(x) = 1, \quad T_1(x) = x, \quad T_k(x) = 2xT_{k-1}(x) - T_{k-2}(x),$$
and have the property \cite{Auzinger2011},
\begin{align}
\underset{p \in \mathbb{P}_k, p(\gamma) = 1}{\argmin}\;\underset{t \in [\alpha,\beta]}{\max} |p(t)| = \frac{1}{T_k\big(1 + \frac{\gamma - \beta}{\beta - \alpha}\big)}T_k\bigg(1 + \frac{t - \beta}{\beta - \alpha}\bigg), \label{eq: cheby_prop}
\end{align}
where $\mathbb{P}_k$ is the set of monic polynomials of degree $k$ with real coefficients.  In other words, the Chebyshev polynomials can be used to construct polynomials with minimal $\ell^\infty$ norms over intervals.  We will use a polynomial which minimizes the maximum distance of the nonzero eigenvalues of $\mathbf L$ to 1, while preserving the eigenvalue at 0,
\begin{align}
\underset{p \in \mathbb{P}_k, p(0) = 0}{\argmin}\;\underset{t \in [\lambda_2,\lambda_n]}{\max} |1 - p(t)|. \label{eq: cheby_cond_prop}
\end{align}
Such a polynomial when applied to $\mathbf L$ will cluster all of the eigenvalues closer to 1, thus reducing the condition number.  This approach is used in \cite{Scaman2017}, where it is shown that for a scaled version of the Laplacian, $c_2\mathbf{L}$, where $c_2 = \frac{2}{(1 + \frac{\lambda_2}{\lambda_n})\lambda_n}$, the relation \eqref{eq: cheby_prop} allows one to find that 
\begin{align}
\underset{p \in \mathbb{P}_k, p(0) = 0}{\argmin}\;\;\underset{t \in [1 - c_1^{-1}, 1 + c_1^{-1}]}{\max} |1 \!- \!p(t)|\! = \!1\! -\! \frac{T_K(c_1(1-x))}{T_K(c_1)}, 
\end{align}
where $c_1 = (1 + \frac{\lambda_2}{\lambda_n})/(1 - \frac{\lambda_2}{\lambda_n})$ and the nonzero eigenvalues of $c_2 \mathbf L$ lie in $[1 - c_1^{-1}, 1 + c_1^{-1}]$.  We similarly propose to precondition the Laplacian matrix with the polynomial,
\begin{align}
P_k(x) = 1 - \frac{1}{T_k(c_1)}T_k\big(c_1(1-x)\big).
\end{align}
We verify that the matrix $P_k(c_2\mathbf{L})$ will still create consensus among agents.  The proof is deferred to the appendix.
\begin{proposition}\label{proposition: kernel_check}
The preconditioned gossip matrix $P_k(c_2 \mathbf L)$ satisfies $\ker P_k(c_2 \mathbf L) = \operatorname{span}\{1_n \otimes x\;|\;x \in \R^d\}$.
\end{proposition}

Thus, we may encode the constraint of the problem \eqref{eq: distributed_problem} as $\sqrt{P_k(c_2\mathbf{L})}\bx = 0$ and form the associated saddle point problem.  Using the recurrence relation for the Chebyshev polynomials we arrive at the following preconditioned primal dual algorithm.
\begin{algorithm}[h]
	{\footnotesize \vspace*{1ex}
		\textbf{Given}: $f = \sum_{i=1}^n f_i$, where $f$ has $L_f$-Lipschitz gradient, $g = \sum_{i=1}^n g_i$.  An undirected connected communication graph with gossip matrix $\operatorname L = [\operatorname L_{ij}]$, Chebyshev polynomial degree $K$, step size $\alpha>0$, augmentation parameter $\rho > 0$, set $\beta = \frac{1}{P_K(c_1\lambda_n)}\bigg(\frac{1}{\alpha} - L_f\bigg)-\rho$.
		\begin{algorithmic}[1]
			\State Initialize at $x_i^0$ such that $x_i^0 = x_j^0$ for all $i,j =1,\ldots,n$, initialize $\nu_i^0$ such that $\sum_i \nu_i^0 = 0$			
			\For {$k=0,1,2,\cdots$ all agents}
			\State 
			$
			\hat x_i^{k+1} = x_i^k -\alpha\big(\nabla f_i(x_i^k) + \nu_i^k).
			$\vspace{.3 mm}
			\State
			$
			x_i^{k+1} = \argmin_x \{g_i(x) + \frac{1}{2\alpha}\|x - \hat x_i^{k+1}\|_2^2\}
			$\vspace{1 mm}
			\LineComment{/* \texttt{Initialization of auxiliary variables for accelerated gossip} */} \vspace{1 mm}
			\State 
			$
			a_i^0 = 1,\;a_i^1 = c_1.
			$
			\State
			$
			\xi_i^0 = (\rho+2\beta)x_i^{k+1}-(\rho + \beta)x_i^k.
			$\vspace{.3 mm}
			\State
			$
			\xi_i^1 = c_1\xi_i^0 - c_1c_2\sum_{\ell \in \mathcal N_i} \operatorname L_{i\ell} \xi_\ell^0.
			$\vspace{1 mm}
			\LineComment{/* \texttt{Accelerated gossip rounds} */} \vspace{1 mm}
			\For {$j=1,2,\ldots,K-1$ all agents} 
			\State 
			$
			a_i^{j+1} = 2c_1a_i^j - a_i^{j-1}
			$\vspace{.3 mm}
			\State
			$
			\xi_i^{j+1} = 2c_1\xi_i^j - 2c_1c_2 \sum_{\ell \in \mathcal N_i} \operatorname L_{i\ell} \xi_\ell^j - \xi_i^{j-1}
			$\vspace{.3 mm}
			\EndFor
			
			\State 
			$
			\nu_i^{k+1}  = \nu_i^k + \bigg(\xi_i^0 - \frac{\xi_i^K}{a_i^K}\bigg).
			$
			\EndFor
	\end{algorithmic}}
	\caption{\hspace*{-.5ex} Chebyshev-Accelerated Distributed Primal-Dual Algorithm} \label{Algorithm: GPDPDA}

\end{algorithm}

Note that, for a given primal step size $\alpha$, we have set the dual step size, $\beta = \frac{1}{P_K(c_1\lambda_n)}\bigg(\frac{1}{\alpha} - L_f\bigg) - \rho$, to be as large as possible (as given by Theorem 1) after the effect of the Chebyshev preconditioning on the Laplacian matrix.  The updates for $\hat x_i^{k+1}$ and $x_i^{k+1}$ are identical to those in Algorithm \ref{Algorithm: Vanilla_PD}.  After these updates we must use the auxiliary variables $\{ \xi_i^j\}_{j=1}^K$  and $\{ a_i^j \}_{j=1}^K$ in order to compute the recurrence for the Chebyshev polynomial.  The net effect of the inner loop is to use a gossip matrix with better spectral properties, at the cost of taking $K$ extra rounds of communication per gradient update.  Note that this effect is not achieved by simply performing the dual update in Algorithm \ref{Algorithm: Vanilla_PD} (the original algorithm without preconditioning) $K$ times per primal update.  Depending on the specific application, one would have to decide how many extra rounds of communication can be tolerated based on the relative costs of gradient updates and communications between agents.  When the cost of agent communications is relatively cheap compared to the cost of gradient updates, the degree of the preconditioning Chebyshev polynomial can be larger without significantly increasing the time it takes for the algorithm to converge.

%% file: Numerical_Simulations_revisions.tex

\section{Numerical Simulations} \label{sec: Numerical Simulations}

In this section, we consider a distributed sparse signal recovery problem with an unknown signal $x_0 \in \mathbb{R}^{d}$ ($d=1024$), which consists of $T=10$ spikes with amplitude $\pm 1$. Consider a network of $n=100$ agents, in which agent $i \in \{1,\cdots,n\}$ measures $b_i = A_ix_0+v_i \in \mathbb{R}^{n_i}$ ($n_i=10$ for all $i$), where $v_i$ is drawn according to the Gaussian distribution $\mathcal{N}(0, 0.01 I_{n_i})$ on $\mathbb{R}^{n_i}$, and $A_i \in \mathbb{R}^{n_i \times d}$ is the measurement matrix ($n_i<d$) of agent $i$. The aggregated measurement matrix $A=[A_1^\top \cdots A_n^\top]^\top \in \mathbb{R}^{1000 \times 1024}$ is obtained by orthagonalizing the rows of a $1000 \times 1024$ matrix whose entries are generated independently and identically according to the standard normal distribution \cite{kim2007interior}. The goal of the agents is to recover the sparsity pattern of $x^\star$ without exchanging their data $(A_i,b_i)$. The corresponding centralized optimization problem can be written as
\begin{align} \label{eq: sparse recovery distributed}
&\mathrm{minimize}_{x \in \mathbb{R}^d} \ &&\sum_{i=1}^{n} \{\underbrace{\frac{1}{2} \|A_ix_i-b_i\|_2^2}_{f_i(x)} \ +\underbrace{ \alpha_i \|x\|_1}_{g_i(x_i)}\}  \\
&\text{subject to} &&x_1=\cdots=x_n, \nonumber
\end{align} 
where $\alpha_i = \alpha/n$, $\alpha=0.01$ for all $i$. The first term is the sum of the squared measurement errors and the second term is an $\ell_1$ regularization to enhance sparsity in the decision variables. 
\begin{figure}
	\centering
	\begin{subfigure}{0.5\textwidth}
		\includegraphics[width=\linewidth]{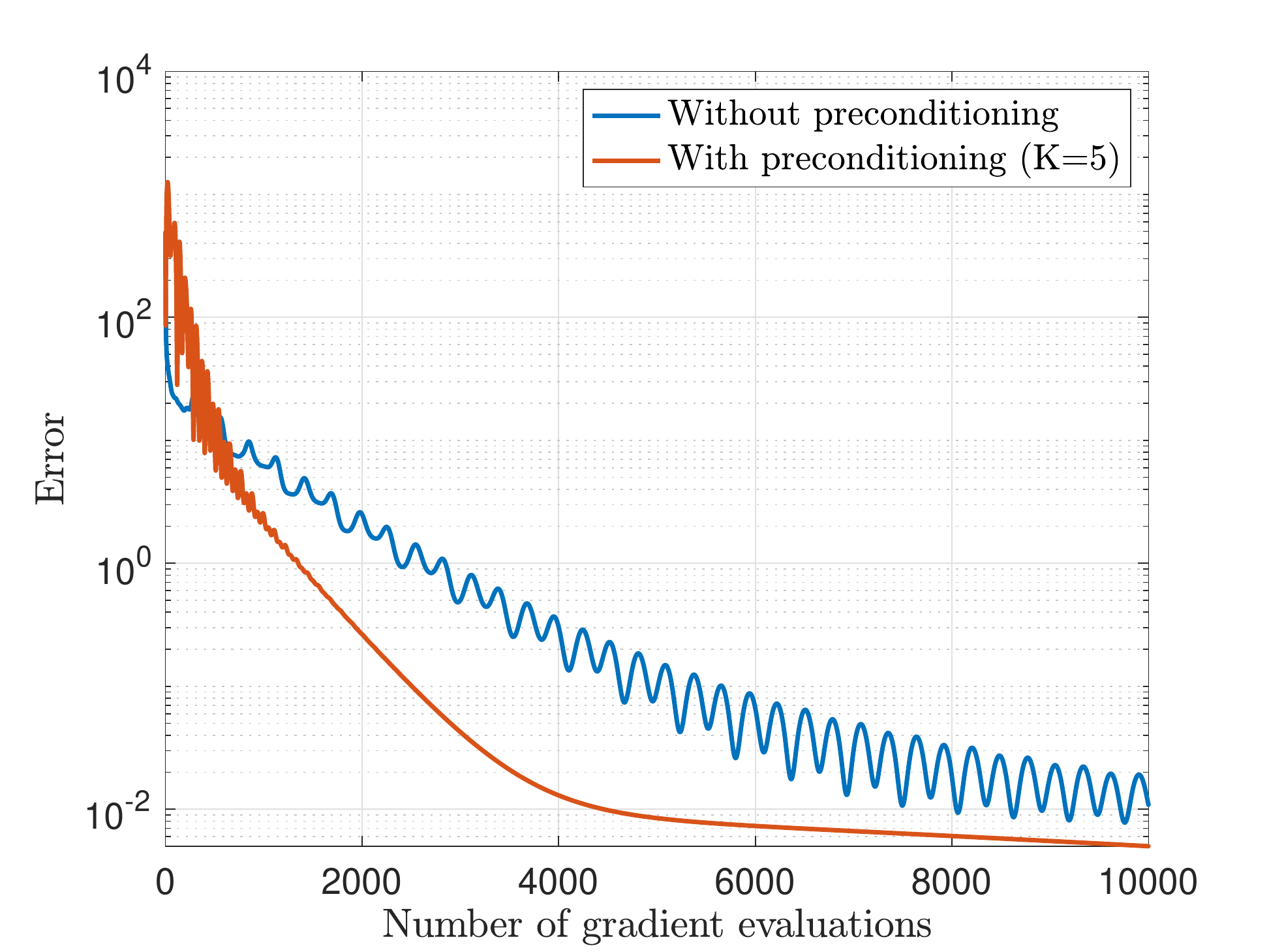}
		\caption{\small Average suboptimality for sparse signal recovery on on a chain graph with $n=100$ nodes.}\label{fig: chain graph error evolution 1}
	\end{subfigure}
	\begin{subfigure}{0.5\textwidth}
		\includegraphics[width=\linewidth]{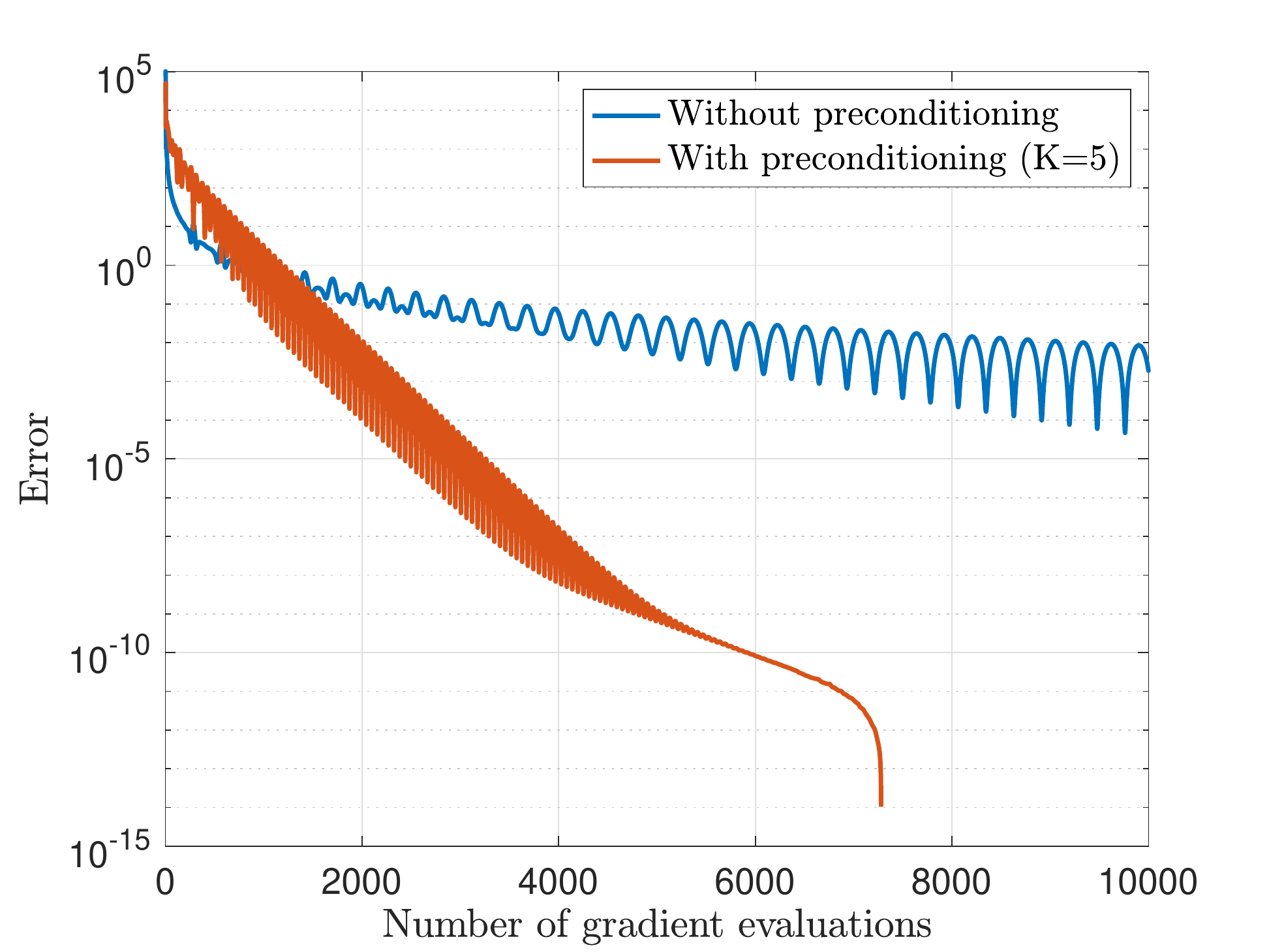}
		\caption{\small Total disagreement between agents on a chain graph with $n=100$ nodes.}\label{fig: chain graph error evolution 2}
	\end{subfigure}
	\caption{}
\end{figure}
Given these numerical values, we run Algorithm \ref{Algorithm: Vanilla_PD} and \ref{Algorithm: GPDPDA} with an appropriate selection of the step sizes prescribed by Theorem  \ref{theorem: convergence_thrm}. To measure the progress of the algorithm, we compute the quantities
\begin{subequations}
\begin{align}
\varepsilon_1^k &= \dfrac{1}{n} \sum_{i=1}^{n} \{(f_i(x^k_i)-f_i(x^*)) + (g_i(x^k_i)-g_i(x^*))\}, \\
\varepsilon_2^k &=\dfrac{1}{2} \sum_{i=1}^{n} \sum_{j \in \mathcal{N}_i} \|x_i^k-x_j^k \|_2^2.
\end{align}
\end{subequations}
where $x^\star$ is an optimal solution to the centralized problem. The first term, $\varepsilon_1^k$, is the average suboptimality, while the second term, $\varepsilon_2^k$, quantifies the disagreement between the agents states. 

In Figure \ref{fig: chain graph error evolution 1}, we plot $\varepsilon_1^k$ and $\varepsilon_2^k$ for Algorithm \ref{Algorithm: GPDPDA} with $K=1$ (no preconditioning) and $K=5$ (preconditioning using five communication rounds per gradient update) applied to a chain network.  Such a network has poor mixing properties, and we clearly observe the benefit of Chebyshev preconditioning on the empirical convergence rate. For instance, without preconditioning, the number of gradient evaluations per node to achieve an accuracy level of $0.1$ is roughly $5000$, while this number is roughly $2500$ after preconditioning.  In Figure \ref{fig: chain graph error evolution 2} we see how far from consensus the network is as the iterations increase.  The high degree of oscillations correspond to the difficulty in maintaining consensus during the iterations for a network with a high condition number, such as a chain graph.

We also repeat the experiment for an Erd\H{o}s-R\'enyi network with an average node degree of $3$ and we plot the error in Figure \ref{fig: random graph error evolution 1}. Here we observe a significant increase the performance of our algorithm compared to the standard primal dual method.  From Figure \ref{fig: random graph error evolution 2} we observe that preconditioning drives the network more quickly to consensus.  

We acknowledge that the added complexity of the additional rounds of communication in the preconditioned case is hidden in these plots.  However, for the case where gradients are much more expensive to compute compared to the communications these plots give a reasonable indication of relative performance of the algorithms.  We emphasize that our algorithm is most useful in such scenarios.

\begin{figure}
	\centering
	\begin{subfigure}{0.5\textwidth}
		\includegraphics[width=\linewidth]{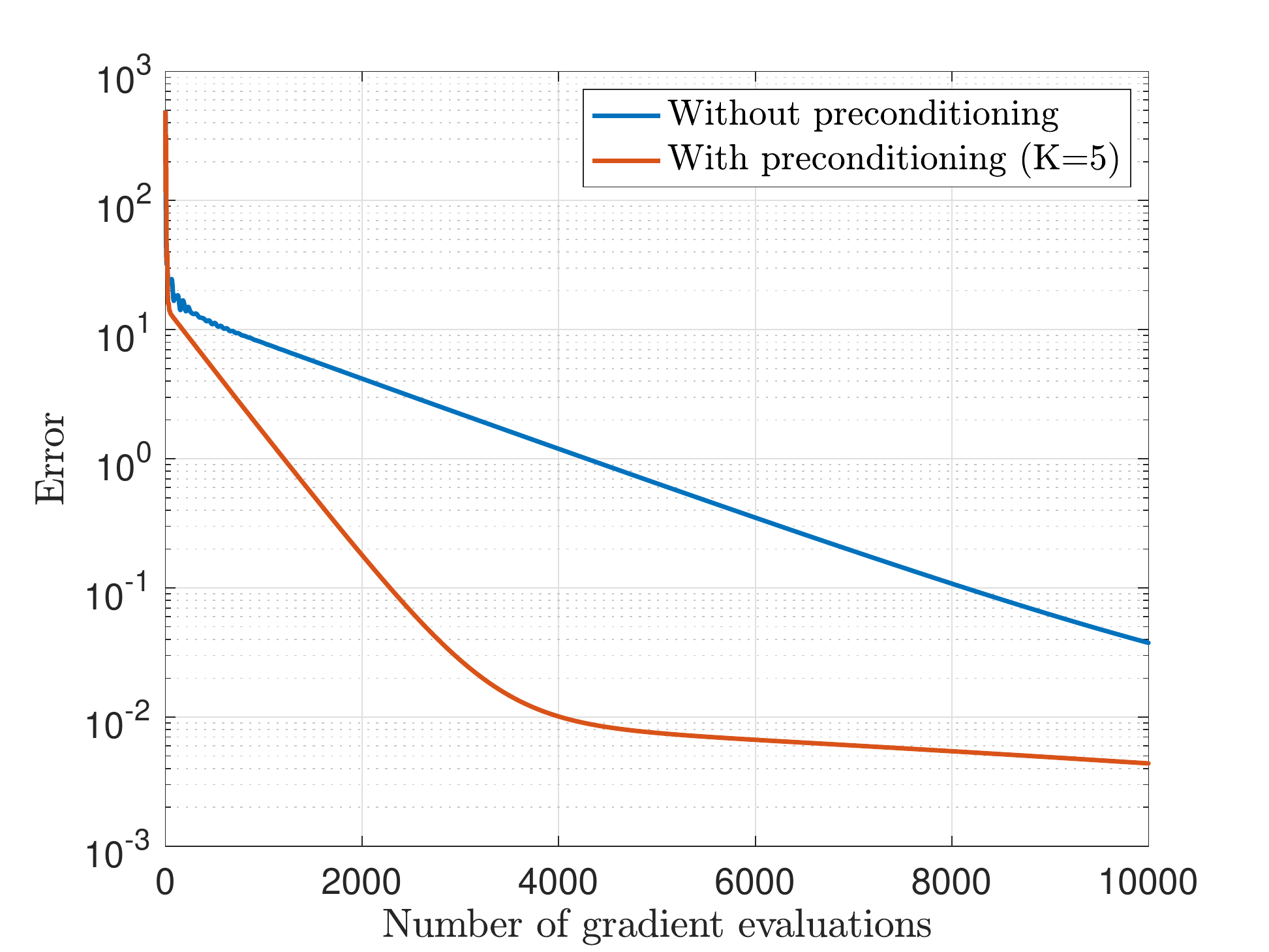}
		\caption{\small Average suboptimality for sparse signal recovery on an\\ Erd\H{o}s-R\'enyi graph with average degree $3$ and $n=100$ nodes.}\label{fig: random graph error evolution 1}
	\end{subfigure}
	\begin{subfigure}{0.5\textwidth}
	\includegraphics[width=\linewidth]{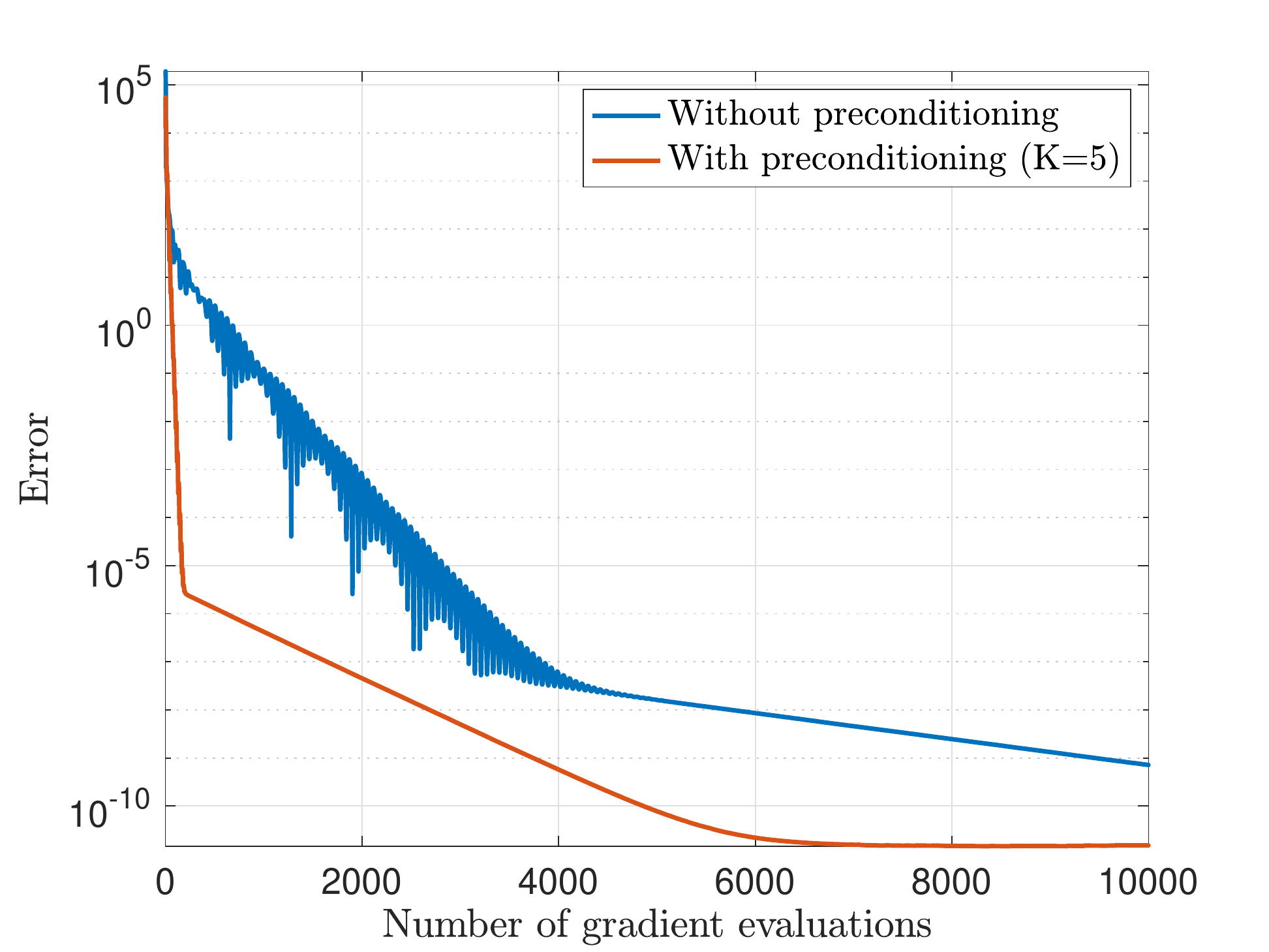}
	\caption{\small Total disagreement between agents on an Erd\H{o}s-R\'enyi graph \\with average degree $3$ and $n=100$ nodes.}\label{fig: random graph error evolution 2}
\end{subfigure}
\caption{}
\end{figure}

%% file: Conclusion_revisions.tex
\section{Conclusion}

We have considered a distributed optimization problem over a network of agents. To this end, our main contribution is a distributed Chebyshev-accelerated primal-dual algorithm where we use Chebyshev matrix polynomials to generate gossip matrices whose spectral properties result in faster convergence speeds, while allowing for a fully distributed implementation. As a result, the proposed algorithm requires fewer gradient updates at the cost of additional rounds of communications between agents. We have illustrated the performance of the proposed algorithm in a distributed signal recovery problem. We have also shown how the use of Chebyshev matrix polynomials can be used to improve the convergence speed over communication networks by trading local computation by communication rounds. For future work, we will explore the possibility of using Chebyshev polynomials to gossip the gradients between updates to obtain a faster convergence. Furthermore, we will study the option of adding (properly tuned) momentum terms to the updates in order to improve the convergence speed of our algorithm.

%% file: appendix.tex
From the definition of $\mathbf L$ it suffices to prove that $\ker P_k(c_2 \operatorname L) = \operatorname{span}\{1_n\}$.  Since $0$ is an eigenvalue of $\operatorname L$ we have that $0$ is an eigenvalue of $P_k(c_2 \operatorname L)$ as well, as $P_K(0) = 0$.  We claim that this eigenvalue has multiplicity 1.  If not, then there is some nonzero eigenvalue $\lambda$ of $c_2 \operatorname L$, such that $P_k(\lambda) = 0$ or equivalently, that $T_k(c_1(1 - \lambda)) = T_k(c_1)$.  

Recall that $c_1 > 1$ and note that the nonzero eigenvalues of $c_2 \operatorname L$ lie in the interval $[1 - c_1^{-1}, 1 + c_1^{-1}]$, so $\lambda$ is contained in this interval.  This implies that $c_1(1 - \lambda) \in [-1,1]$.  Since $|T_k(x)| > 1$ when $|x| > 1$ and $|T_k(x)| \leq 1$ otherwise, we cannot have $T_k(c_1(1-\lambda)) = T_k(c_1)$ unless $\lambda = 0$.  Thus, $\lambda = 0$ has multiplicity 1 as an eigenvalue of $P_k(c_2 \operatorname L)$.

It remains to show that for $x \in \operatorname{span}\{1_n \}$ we have $P_k(c_2 \operatorname L)x = 0$.  From the definition of $P_k$, it suffices to prove that
$$T_k(c_1(\operatorname I_n - c_2 \operatorname L)) = T_k(c_1)x.$$
We proceed inductively using the recurrence relation for $T_k$.  Assuming that $T_{k-1}(c_1(\operatorname I_n - c_2 \operatorname L)) = T_{k-1}(c_1)x$ and using the fact that $\operatorname L x = 0$, we see
\begin{align}
T_k(c_1(\operatorname I_n - c_2 \operatorname L))x &= 2c_1(\operatorname I_n - c_2 \operatorname L)T_{k-1}(\operatorname I_n - c_2 \operatorname L)x \nonumber\\
 &\quad - T_{k-2}(\operatorname I_n - c_2 \operatorname L)x \nonumber\\
&=2c_1(\operatorname I_n \!- \!c_2 \operatorname L)T_{k-1}(c_1)x \!-\! T_{k-2}(c_1)x \nonumber\\
&= (2c_1 T_{k-1}(c_1) - T_{k-2}(c_1))x \nonumber\\
&= T_k(c_1)x.\nonumber
\end{align}
Hence, $\operatorname{span}\{1_n \}\subseteq \ker P_k(c_2 \operatorname L)$, and since we have shown $\ker P_k(c_2 \operatorname L)$ is one dimensional, we conclude $\ker P_k(c_2 \operatorname L) = \operatorname{span}\{1_n \}$, as desired. $\blacksquare$
